\documentclass{amsart}

\usepackage{ae}
\usepackage[all]{xy}
\usepackage{graphicx,amsfonts,amssymb,amsmath}
\usepackage{natbib}
\usepackage{enumerate}

\usepackage[latin1]{inputenc}
\usepackage[T1]{fontenc}
\usepackage[english]{babel}
\usepackage[cyr]{aeguill}

\newcommand{\chapeau}{{\rlap{\smash{\hbox{\lower4pt\hbox{\hskip1pt$\widehat{\phantom{u}}$}}}}}\mbox{ }}
\usepackage[OT2,T1]{fontenc}
\DeclareSymbolFont{cyrletters}{OT2}{wncyr}{m}{n}
\DeclareMathSymbol{\sha}{\mathalpha}{cyrletters}{"58}

 \newtheorem{thm}{Theorem}[section]

 \newtheorem{thm'}[thm]{Theorem'}

 \newtheorem*{pb*}{\textit{Open question}}
 \newtheorem*{cor*}{\textit{Corollary}}

 \newtheorem{cor}[thm]{Corollary}
 \newtheorem{lem}[thm]{Lemma}
 
 \theoremstyle{definition}
 \newtheorem{defn}[thm]{Definition}
 \theoremstyle{definition}
 \newtheorem*{ter}{Terminology}
 \theoremstyle{remark}
 
 \theoremstyle{remark}
 \newtheorem{rem}[thm]{Remark}
 \theoremstyle{remark}

 \numberwithin{equation}{subsection}

 \renewcommand{\P}{\mathbb{P}}
 \newcommand{\Q}{\mathbb{Q}}
 \newcommand{\Z}{\mathbb{Z}}
 \newcommand{\G}{\mathbb{G}}
 \newcommand{\K}{\mathcal{K}}
 \newcommand{\Hil}{\mathsf{Hil}}

 \newcommand{\Gal}{\textup{Gal}}
 \newcommand{\coker}{\textup{Coker}}
 \newcommand{\cores}{\textup{cores}}
 \renewcommand{\ker}{\textup{Ker}}
 \newcommand{\Br}{\textup{Br}}
 \newcommand{\Pic}{\textup{Pic}}
 \newcommand{\CH}{\textup{CH}}
 \renewcommand{\H}{\textup{H}}
 \newcommand{\Hom}{\textup{Hom}}
 \newcommand{\Spec}{\textup{Spec}}
 
 \newcommand{\inv}{\textup{inv}}
 \newcommand{\E}{\textup{E}}

\usepackage{hyperref}

\begin{document}

\title[]
{Local-global principle for 0-cycles on fibrations over rationally connected bases}

\author{ Yongqi LIANG  }

\address{Yongqi LIANG
\newline B\^atiment Sophie Germain,
\newline Universit\'e Paris Diderot - Paris 7,
\newline Institut de Math\'ematiques de Jussieu -  Paris Rive Gauche,
 \newline  75013 Paris,\newline
 France}

\email{liangy@math.jussieu.fr}

\thanks{\textit{Key words} : zero-cycles,  weak approximation,
Brauer\textendash Manin obstruction}

\thanks{\textit{MSC 2010} : 11G35 (14G25, 14D10)}

\date{\today.}



\maketitle

\begin{abstract}
We study the Brauer\textendash Manin obstruction to the Hasse principle and to weak approximation for 0-cycles on algebraic varieties that possess a fibration structure. The exactness of the local-to-global sequence $(\E)$ of Chow groups of 0-cycles was known only for a fibration whose base is either a curve or the projective space. In the present paper, we prove the exactness of $(\E)$ for fibrations whose bases are Ch\^{a}telet surfaces or projective models of homogeneous spaces of connected linear algebraic groups with connected stabilizers. 
We require that either all fibres are split and most fibres satisfy weak approximation for 0-cycles, or the generic fibre has a 0-cycle of degree $1$ and $(\E)$ is exact for most fibres.
\end{abstract}

\tableofcontents

\section{Introduction}

\subsection{Background}
Let $k$ be a number field. In this paper, we consider algebraic varieties $X$ defined over $k$. It is conjectured by Colliot-Th\'el\`ene\textendash Sansuc and Kato\textendash Saito that the Hasse principle and weak approximation for 0-cycles on proper smooth varieties can be controlled by their Brauer groups. To be precise, we are interested in the exactness of the following complex which will be explained with more details in \S \ref{notationsection}
$$\varprojlim_{n}\CH_{0}(X)/n\to\prod_{v\in\Omega_{k}}\varprojlim_{n}\CH'_{0}(X_{k_v})/n\to\Hom(\Br X,\Q/\Z).\leqno(\E)$$

The exactness of $(\E)$ for $X=\Spec(k)$ (or $X=\P^{m}$) is ensured by global class field theory. The exactness of $(\E)$ for smooth projective curves was proved by  Saito \cite{Saito} and \cite{CT99HP0-cyc} under the assumption of the finiteness of Tate\textendash Shafarevich groups of their Jacobians.
Concerning higher dimensional varieties, known results are only available for several homogeneous spaces of linear algebraic groups and varieties possessing a fibration structure. Here we consider only fibrations $X\to B$. Existing results have dealt with fibrations whose base $B$ is
\begin{itemize}
\item[-] either a  curve $C$ (with finiteness of $\sha(\textup{Jac}_{C,k})$ assumed), 
\item[-] or the projective space $\P^{m}$  by induction reducing to $m=1$.
\end{itemize}
See the very recent paper of Harpaz and Wittenberg \cite{HarWit} for  a summary.

Concerning analogous questions for rational points, the fibration method allows one to deal with other bases besides curves and projective spaces, for example \cite[Thm. 4.3.1]{Harari} and \cite[Thm. 3]{HarariSMF}. However, the arguments cannot be extended directly to 0-cycles. The main difficulty was that on the base $B$ we did not know how to approximate topologically a family of effective local 0-cycles in good position by a single global closed point even if $B$ satisfies weak approximation for rational points (over any finite field extension). When $B$ is $\P^{1}$, this is trivial for rational points; and for 0-cycles this can be done by Salberger's device \cite{Salberger}.  More generally, if $B$ is a smooth projective curve whose jacobian has finite Tate\textendash Shafarevich group, this can be treated by Colliot-Th\'el\`ene's argument for 0-cycles, \cite{CT99}. We are trying to overcome this difficulty at least for geometrically rationally connected varieties.

\subsection{Our results}
In the author's Ph.D. dissertation defence, Per Salberger asked  whether one can prove results for fibrations whose  base is a Grassmannian. This work is a first try to answer this question, actually we prove much more general results which can certainly be applied when $B$ is a Grassmannian. Our main results are as follows, we refer to \S \ref{notationsection} for  relevant   terminology. We will give proofs in \S \ref{proofsection} of stronger and more detailed results \textemdash~ Theorems \ref{mainThm-geo-int} and \ref{mainThm-section}.

\begin{thm}[\textit{cf.} Theorem \ref{mainThm-geo-int}]\label{simpleThm-geo-int}
Let $f:X\to B$ be a dominant morphism between proper smooth geometrically rationally connected varieties defined over  a number field $k$. Assume that its geometric generic fibre is integral.

Suppose that
\begin{enumerate}
\item[1,]  the fibration $f$ is split in codimension  $1$ (Definition \ref{codim1split});
\item[2,]  for all finite extensions $K$ of $k$, all smooth fibres of $f_{K}:X_{K}\to B_{K}$ over $K$-rational points of $B_{K}$ verify weak approximation  for 0-cycles of degree $1$;
\item[3,]  for all finite extension $K$ of $k$, the Brauer\textendash Manin obstruction is the only obstruction to  weak approximation for $K$-rational points on $B_{K}$.
\end{enumerate}
Then the sequence $(\E)$ is exact for $X$.
\end{thm}

\begin{thm}[\textit{cf.} Theorem \ref{mainThm-section}]\label{simpleThm-section}
Let $f:X\to B$ be a  dominant morphism between proper smooth geometrically rationally connected varieties defined over  a number field $k$. Assume that its geometric generic fibre is integral.

Suppose that
\begin{enumerate}
\item[1,] the generic fibre is geometrically rationally connected and possesses a 0-cycle of degree $1$;
\item[2,]  for all finite extensions $K$ of $k$, the sequence $(\E)$ is exact for all smooth fibres of $f_{K}:X_{K}\to B_{K}$ over $K$-rational points of $B_{K}$;
\item[3,]  for all finite extension $K$ of $k$, the Brauer\textendash Manin obstruction is the only obstruction to weak approximation for $K$-rational points on $B_{K}$.
\end{enumerate}
Then the sequence $(\E)$ is exact for $X$.
\end{thm}

By considering the usual external product of cycles, it is clear that the Brauer\textendash Manin obstruction is the only obstruction to the Hasse principle for 0-cycles of degree $1$ on $X\times Y$ if it is the case for both $X$ and $Y$. Surprisingly, for weak approximation properties, the following immediate consequence of the theorems was not known. Even in this most trivial case, arguments for the analogue  on rational points do not extend directly to 0-cycles and in our proof we need to impose rational connectedness on the varieties.

\begin{cor}\label{trivialcor}
Let $X,Y$ be geometrically rationally connected varieties defined over a number field $k$.

Consider the following conditions  for all finite extension $K$ of $k$
\begin{enumerate}
\item[1,]  the Brauer\textendash Manin obstruction is the only obstruction to weak approximation for $K$-rational points on $X_{K}$;
\item[2,]  $Y_{K}$ verifies weak approximation for 0-cycles of degree $1$;
\item[2',] the sequence $(\E)$ is exact for $Y_{K}$.
\end{enumerate}
If conditions 1 and 2 are satisfied, then the sequence $(\E)$ is exact for $X\times Y$.\ \\
If conditions 1 and 2' are satisfied, and assume furthermore that $Y$ possesses a 0-cycle of degree $1$, then  the sequence $(\E)$ is exact for $X\times Y$.
\end{cor}

Besides trivial fibrations, we have several applications  of the detailed Theorems \ref{mainThm-geo-int}, \ref{mainThm-section}, we mention the following one which follows directly from  Theorem \ref{simpleThm-section}. We refer to \S \ref{appl-subsection} for more details and comments.
\begin{cor}
Let $B$ be a smooth compactification of a homogeneous space of a connected linear algebraic group with connected stabilizer defined over a number field $k$. Let $X\to B$ be a proper dominant morphism whose generic fibre is birationally defined by the polynomial
$$N_{L|k(B)}(\textup{x})=P(t),$$
where $P(t)\in k(B)[t]$ is a separable polynomial of degree prime to the degree of the finite extension $L$ of $k(B)$.

Then $(\E)$ is exact for $X$.
\end{cor}

\section{Main results and applications}\label{sectionmainresult}

\subsection{Notation and terminology}\label{notationsection}
We begin with fixing the notation.  The base field $k$ will be a number field unless otherwise stated. The set of its places is denoted by $\Omega_{k}$ and $S$ denotes always a finite subset of $\Omega_{k}$. Let $k'$ be a finite extension of $k$, the set of places of $k'$ lying above places in $S$ will be denoted by $S\otimes_{k}k'$. We will consider algebraic varieties (separated schemes of finite type) $X$ defined over $k$. Its cohomological Brauer group is  $\Br X=\H^{2}_{\scriptsize{\textup{\'et}}}(X,\mathbb{G}_{m})$, and its Chow group of 0-cycles is denoted by $\CH_{0}(X)$. The base change $X\times_{\Spec(k)}\Spec(K)$ is written simply $X_{K}$ for any extension $K$ of $k$. The modified local Chow group $\CH'_{0}(X_{k_v})$ is defined to be the usual Chow group $\CH_{0}(X_{k_v})$ if $v$ is a non-archimedean  place, and otherwise $\coker\left[N_{\bar{k}_{v}|k_{v}}:\CH_{0}(X_{\bar{k}_{v}})\to\CH_{0}(X_{k_{v}})\right]$, in particular it is $0$ if $v$ is a complex place.

Let $A$ be an abelian group, for non-zero integer $n$ we denote by $A/n$ the cokernel $\coker\left[n:A\to A\right]$ of the multiplication by $n$.
Let $L/K$ be a Galois extension of Galois group $\Gal_{L|K}$ and $M$ be a continuous  $\Gal_{L|K}$-module, we denote its Galois cohomology groups by $\H^{i}(L|K,M)$. If $L=\bar{K}$ is a separable closure of $K$ we will write $\H^{i}(K,M)$ for short.

Before the statement of the main results, we recall some terminology.

Let $X$ be a proper smooth geometrically integral variety defined over $k$.
By extending Manin's well-known pairing \cite{Manin}
$$\prod_{v\in\Omega_{k}}X(k_{v})\times\Br X\to\Q/\Z$$
one can define a natural pairing by evaluating  cohomology classes at closed points of $X_{k_v}$
$$\prod_{v\in\Omega_{k}}\CH_{0}(X_{k_v})\times\Br X\to\Q/\Z,$$ which factorizes through
$$\prod_{v\in\Omega_{k}}\CH'_{0}(X_{k_v})\times\Br X\to\Q/\Z,$$
see for example \cite[page 53]{CT95} or \cite[\S 1.1]{Wittenberg} for a detailed definition. This induces a local-to-global complex
$$\CH_{0}(X)\to\prod_{v\in\Omega_{k}}\CH'_{0}(X_{k_v})\to\Hom(\Br X,\Q/\Z).$$
Since $X$ is regular, its Brauer group is torsion, one deduces the complex mentioned in the introduction
$$\varprojlim_{n}\CH_{0}(X)/n\to\prod_{v\in\Omega_{k}}\varprojlim_{n}\CH'_{0}(X_{k_v})/n\to\Hom(\Br X,\Q/\Z),\leqno(\E).$$
Its exactness means roughly that the obstruction to the local-global principle is controlled by the Brauer group.

\begin{ter}
Let $\delta\in\Z$ be an integer. Let $X$ be a smooth projective geometrically integral variety defined over a number field $k$.

For a family of algebraic varieties, the \emph{Hasse principle for 0-cycles of degree $\delta$} says that the existence of a family of local 0-cycles of degree $\delta$  implies the existence of a global 0-cycle of degree $\delta$. We say that the \emph{Brauer\textendash Manin obstruction is the only obstruction to the Hasse principle for 0-cycles of degree $\delta$ on $X$}, if the existence of local families of 0-cycles of degree $\delta$ orthogonal to $\Br X$ implies the existence of a global 0-cycle of degree $\delta$.

We say that $X$ verifies \emph{weak approximation for 0-cycles of degree $\delta$} if the following statement is satisfied.
\begin{itemize}
\item[]
Given a family of local 0-cycles $\{z_{v}\}_{v\in\Omega_{k}}$ of degree $\delta$,
for any positive integer $n$ and for any finite subset $S$ of $\Omega_{k}$, there exists a global 0-cycle $z=z_{n,S}$ of degree $\delta$ such that
$z$ and $z_{v}$ have the same image in $\CH_{0}(X_{k_v})/n$ for all $v\in S.$
\end{itemize}
We say that  the \emph{Brauer\textendash Manin obstruction is the only obstruction to weak approximation for 0-cycles of degree $\delta$ on $X$}, if the following statement is satisfied.
\begin{itemize}
\item[]
Given a family of local 0-cycles $\{z_{v}\}_{v\in\Omega_{k}}$ of degree $\delta$ orthogonal to $\Br X,$
for any positive integer $n$ and for any finite subset $S$ of $\Omega_{k}$, there exists a global 0-cycle $z=z_{n,S}$ of degree $\delta$ such that
$z$ and $z_{v}$ have the same image in $\CH_{0}(X_{k_v})/n$ for all $v\in S.$
\end{itemize}

We say that $X$ verifies \emph{weak approximation for rational points} if the diagonal image of $X(k)$ in $\prod_{v\in\Omega_{k}}X(k_{v})$ is dense.
We say that the  \emph{Brauer\textendash Manin obstruction is the only obstruction to weak approximation for rational points} if the closure of the diagonal image of $X(k)$ in $\prod_{v\in\Omega_{k}}X(k_{v})$ equals to the subset consisting of local points that are orthogonal to  $\Br X$. See \cite[\S 5]{Skbook} or the survey  \cite[2.10]{Peyre} for more details.
\end{ter}

\begin{rem}\mbox{}

\begin{enumerate}
\item The exactness of $(\E)$ implies that the Brauer\textendash Manin obstruction is the only obstruction to the Hasse principle for 0-cycles of degree $1$, \cite[Rem. 1.1(iii)]{Wittenberg}.
\item The exactness of $(\E)$ implies that Brauer\textendash Manin obstruction is the only obstruction to weak approximation for 0-cycles of degree $\delta$ if $X$ possesses  a global 0-cycle of degree $1$, \cite[Prop. 2.2.1]{Liang4}.
\item The exactness of $(\E)$ implies that Brauer\textendash Manin obstruction is the only obstruction to weak approximation for 0-cycles of degree $1$, this follows from the previous two.
\end{enumerate}
\end{rem}

\begin{defn}
Let $B$ be a geometrically integral $k$-variety. A subset $\Hil$ of closed points of $B$ is called a \emph{generalized Hilbertian subset} if there exists a finite \'etale morphism $Z\buildrel{\rho}\over\to U\subset B$ where $Z$ is an integral $k$-variety and $U$ is a nonempty open subvariety of $B$ such that
$$\Hil=\{\theta\in U|\theta\mbox{ is a closed point of }B\mbox{ such that the fibre }\rho^{-1}(\theta)\mbox{ is connected}\}.$$
\end{defn}

\begin{defn}\label{codim1split}
Consider a dominant morphism $f:X\to B$ between proper smooth geometrically integral $k$-varieties, whose generic fibre $X_{\eta}$ is geometrically integral. We say that  the fibration $f$ is \emph{split in codimension  $1$} if the following condition is satisfied.
\begin{enumerate}
\item[-]
For any discrete valuation ring $R$ containing $k$ and whose fraction field equals to the function field $k(B)$, there exists an integral regular $R$-model $\mathcal{X}$ flat and proper over $R$ whose generic fibre is $k(B)$-isomorphic to $X_{\eta}$, such that its special fibre has an irreducible component of multiplicity $1$ which is geometrically integral.
\end{enumerate}
\end{defn}

By \cite[Cor. 1.2]{Sk2}, this property does not depend on the choice of a particular integral regular model, this is a property of the generic fibre.

\subsection{Statements of the main results}

The main results of this paper concern the exactness of $(\E)$ for fibrations over a base of dimension $m>1$ which is not required to be birational to $\P^{m}$ over $k$.

For a fixed finite extension $L$ of $k$, we denote by $\mathcal{K}_{L}$ the set of finite extensions $K$ of $k$ that are linearly disjoint from $L$, \textit{i.e.} $L\otimes_{k}K$ is a field.

\begin{thm}\label{mainThm-geo-int}
Let $f:X\to B$ be a dominant morphism between proper smooth geometrically integral varieties defined over  a number field $k$. Assume that its geometric generic fibre is integral.

Suppose that
\begin{enumerate}
\item[1,]  the fibration $f$ is split in codimension  $1$;
\item[2,]  for each $K\in\K_{L}$, there exists a generalized Hilbertian subset $\Hil_{K}$ of $B_{K}$ such that for all $K$-rational points $\theta$ contained in $\Hil_{K}$  the fibre $X_{K\theta}=f_{K}^{-1}(\theta)$ verifies weak approximation (resp.  the Hasse principle) for 0-cycles of degree $1$;
\item[3.1,] $\H^{1}(B_{\bar{k}},\mathcal{O}_{B_{\bar{k}}})=0$, $\H^{2}(B_{\bar{k}},\mathcal{O}_{B_{\bar{k}}})=0$, and the N\'eron\textendash Severi group $\textup{NS}B_{\bar{k}}$ is torsion-free;
\item[3.2,]  for all $K\in\K_{L}$, the Brauer\textendash Manin obstruction is the only obstruction to  weak approximation for $K$-rational points on $B_{K}$.
\end{enumerate}
Then the Brauer\textendash Manin obstruction is the only obstruction to weak approximation (resp. to the Hasse principle) for 0-cycles of degree $\delta$  on $X$.

Moreover if
\begin{enumerate}
\item[4,] $\Pic X_{\bar{k}}$ is torsion-free, and $\deg:\CH_{0}(X\otimes\overline{k(X)})\to\Z$ is injective,
\end{enumerate}
then the sequence $(\E)$ is exact for $X$.
\end{thm}

\begin{thm}\label{mainThm-section}
Let $f:X\to B$ be a  dominant morphism between proper smooth geometrically integral varieties defined over  a number field $k$. Assume that its geometric generic fibre is integral.

Suppose that
\begin{enumerate}
\item[1.1,] for the geometric generic fibre $X_{\bar{\eta}}=X_{\eta}\otimes \overline{k(B)}$,  the N\'eron\textendash Severi group $\textup{NS}X_{\bar{\eta}}$ is torsion-free, and $\H^{1}(X_{\bar{\eta}},\mathcal{O}_{X_{\bar{\eta}}})=0$, $\H^{2}(X_{\bar{\eta}},\mathcal{O}_{X_{\bar{\eta}}})=0$.
\item[1.2,]  the generic fibre $X_{\eta}$ viewed as a $k(B)$-variety possesses  a 0-cycle of degree $1$;
\item[2,]  for each $K\in\K_{L}$, there exists a generalized Hilbertian subset $\Hil_{K}$ of $B_{K}$ such that for all $K$-rational points $\theta$ contained in $\Hil_{K}$ the Brauer\textendash Manin obstruction is the only obstruction to weak approximation (resp. to the Hasse principle) for 0-cycles of degree $1$  on the fibre $X_{K\theta}=f_{K}^{-1}(\theta)$;
\item[3.1,]   $\H^{1}(B_{\bar{k}},\mathcal{O}_{B_{\bar{k}}})=0$, $\H^{2}(B_{\bar{k}},\mathcal{O}_{B_{\bar{k}}})=0$, and the N\'eron\textendash Severi group $\textup{NS}B_{\bar{k}}$ is torsion-free;
\item[3.2,]  for all $K\in\K_{L}$, the Brauer\textendash Manin obstruction is the only obstruction to weak approximation for $K$-rational points on $B_{K}$.
\end{enumerate}
Then the Brauer\textendash Manin obstruction is the only obstruction to weak approximation (resp. to the Hasse principle) for 0-cycles of degree $\delta$   on $X$.

Moreover if
\begin{enumerate}
\item[4,] $\Pic X_{\bar{k}}$ is torsion-free, and $\deg:\CH_{0}(X\otimes\overline{k(X)})\to\Z$ is injective,
\end{enumerate}
then the sequence $(\E)$ is exact for $X$.
\end{thm}

\subsection{Several remarks on the theorems}

\begin{rem}
The analogue, in the context of rational points, of Theorem \ref{mainThm-geo-int} is just the basic idea of the fibration method.
Theorem \ref{mainThm-section} is an analogue of the theorem for rational points \cite[Thm. 3]{HarariSMF}. If furthermore the base $B$ is supposed to satisfy weak approximation, we also refer to \cite[Thm. 4.3.1]{Harari} which is the relatively easy case in the paper \cite{Harari}. Similar to the hypothesis 1.2, there the generic fibre $X_{\eta}$ is supposed to possess a $k(B)$-rational point.
But one does not need the hypothesis 3.1, neither the finiteness of $\Br B/\Br k$.
\end{rem}

\begin{rem}\label{X=B}
In the particular case where $X=B$ and $f:B\to B$ is the identity morphism, both theorems reduce to \cite[Thm. 2.1]{Liang5} \textemdash~ a stronger version of the main result of \cite{Liang4}.
\end{rem}

\begin{rem}\label{remRC}
Let $V$ be a variety defined over algebraically closed field of characteristic $0$. If $\H^{1}(V,\mathcal{O}_{V})=0$, $\H^{2}(V,\mathcal{O}_{V})=0$ and if the N\'eron-Severi group $\textup{NS}(V)$ is torsion-free, then $\Pic V$ is torsion-free and $\Br V$ is finite, in particular it is the case if $V$ is a rationally connected variety. Moreover if the base $B$ and the generic fibre $X_{\eta}$ are both geometrically rationally connected, then so is $X$ according to \cite[Cor. 1.3]{GHS},  hypotheses 1.1, 3.1, and 4, are automatically satisfied.
\end{rem}

\begin{rem}
In both of the theorems, the arithmetic hypothesis concerning the Brauer\textendash Manin obstruction of the  base $B$ is a property of rational points, it is not clear whether it may  be replaced by the exactness of $(\E)$ for $B$ (which is a consequence of the property of rational points  by \cite[Thm. 2.1]{Liang5}). However, the arithmetic hypothesis on fibres concerns the Hasse principle or weak approximation for 0-cycles of degree $1$, which may be valid even if the similar question for rational points is still open (or solved conditionally).
\end{rem}

\begin{rem}
In the theorems, we suppose that certain arithmetic properties are satisfied for extensions $K\in\mathcal{K}_{L}$ instead of for all finite extensions, this may be useful in practice. As an example, we consider a proper smooth model $B$ of the affine variety defined by the equation 
$$N_{k'|k}(\textup{x})=Q(t),$$
where $Q(t)\in k[t]$ is a quadratic irreducible polynomial and $k'$ is a field extension of $k$ of degree $4$ splitting $Q(t)$. One may assume that $k'$ and $k[t]/(Q(t))$ are subfields of a fixed algebraic closure of $k$. Let $L$ be the composition of these two subfields. 
According to \cite[Th. 1]{DSW}, for any $K\in\mathcal{K}_{L}$ the Brauer\textendash Manin obstruction is the only obstruction to weak approximation for rational points on $B_{K}$.

Consider  Ch\^atelet surfaces defined by $x_{1}^{2}-ax_{2}^{2}=P(t)$ with $P(t)\in k[t]$ an irreducible polynomial of degree $4$,  then the Brauer\textendash Manin obstruction is the only obstruction to weak approximation for 0-cycles of degree $1$ by \cite[Thm. 0.4, 0.5]{Frossard}, and its Brauer group is trivial modulo constant by \cite[Prop. 1]{Sansuc82ChateletBr0}. This property is preserved after a finite extension of $k$ linearly disjoint from $k[t]/(P(t))$. If  fibres of a given fibration are such surfaces, then the  hypothesis 2 of Theorem \ref{mainThm-geo-int}, with the restriction to linear disjointness or to Hilbertian subsets,  is satisfied.
Similar restrictions of the hypothesis 2 of Theorem \ref{mainThm-section} may also be useful if fibres of a fibration are varieties considered in \cite[Cor. 3.3, Thm. 4.1]{Wei}.

\end{rem}

\begin{rem}
Theorem \ref{mainThm-section} has a variant concerning the algebraic part of Brauer groups:
\begin{itemize}
\item[-]
we can weaken the geometric hypothesis 1.1, suppose only that $\textup{NS}X_{\bar{\eta}}$ is torsion-free and $\H^{1}(X_{\bar{\eta}},\mathcal{O}_{X_{\bar{\eta}}})=0$ (equivalently $\Pic X_{\bar{\eta}}$ is torsion-free);
\item[-]
we have to strengthen the arithmetic hypothesis 2, suppose that the \emph{algebraic} Brauer\textendash Manin obstruction is the only obstruction for the fibre $X_{K\theta}=f_{K}^{-1}(\theta)$ mentioned above;
\end{itemize}
then the conclusion stays the same. For the proof, we only need to replace the forthcoming Lemma \ref{lemBr} by the comparison of the algebraic part of Brauer groups which is even easier.
 The similar comment applies to the base $B$, or applies also to both the fibres and the base simultaneously.
\end{rem}

\begin{rem}\label{rmksplitness}
Theorem \ref{mainThm-geo-int} remains valid if the geometric hypothesis 1  is replaced by the following arithmetic hypothesis
\begin{enumerate}
\item[$1'$,] There exists a finite set $S$ of places of $k$ and a non-empty open set $B^{0}$ of $B$, such that  for any closed point $\theta$ of $B^{0}$ the fibre $X_{\theta}$ \textemdash~ a $k(\theta)$-variety \textemdash~possesses a local 0-cycle of degree $1$ over $k(\theta)_{w}$ for any place $w$ of $k(\theta)$ lying above a certain place $v\notin S$.
\end{enumerate}
It is clear that $1'$ is automatically satisfied under one of the following conditions.
\begin{enumerate}
\item The generic fibre $X_{\eta}$ of $f$ possesses a 0-cycle of degree $1$.
\item There exists a finite set $S$ of places of $k$ such that for every $v\notin S$ and for every finite extension $F$ of $k_{v}$ the mapping $X(F)\to B(F)$ is surjective.
\end{enumerate}
In his paper \cite{CTfibresp}, Colliot-Th\'el\`ene conjectured that $f:X(k_{v})\to B(k_{v})$ is surjective for almost all places $v$ of $k$ if the fibration $f$ is split in codimension  $1$. The special case where $B$ is $\P^{1}$ is well-known, \cite[pp. 208-209]{Skorobogatov} \cite[Lem. 2.3]{Sk2}. Denef gave two proofs of the conjecture for general case in his preprint \cite{Denef},  one is an algebraic geometric proof and the other uses mathematical logic. The main result of \cite{Denef} is stated for varieties over the field $k=\mathbb{Q}$. As is remarked in the paper,  same proofs are valid for any number field $k$.
In the geometric proof, a finite set $S$ of exceptional places is constructed from and depends only on the geometric data $(X,B,f)$ over $k$. The surjectivity of $X(k_{v})\to B(k_{v})$ for $v\notin S$ comes essentially from the Lang\textendash Weil estimate and Hensel's lemma  via a diophantine purity theorem with sophisticated arguments. If one checks every step of the proof, with the same finite set $S$, for $v\notin S$ his argument also proves the surjectivity of $X(F)\to B(F)$ for any finite extension $F$ of $k_{v}$, \textit{i.e.} the condition (2). Hence the arithmetic hypothesis $1'$ is a consequence of the geometric hypothesis 1.

We will give a complete proof of Theorem \ref{mainThm-section} in \S\ref{proofsection}. The proof of Theorem \ref{mainThm-geo-int} is easier and it shares a  large part in common with the proof of Theorem \ref{mainThm-section}, a sketch (with hypothesis 1 replaced by hypothesis $1'$ or (2) ) will be given after the  proof of Theorem \ref{mainThm-section}.
\end{rem}

\subsection{Applications}\label{appl-subsection}

In this section, we give several applications of Theorems \ref{mainThm-geo-int} and \ref{mainThm-section}. All varieties given here are geometrically rationally connected, the hypotheses 1.1, 3.1 and 4 are automatically satisfied after Remark \ref{remRC}.

In both theorems, the base $B$ can be any smooth proper $k$-variety birationally equivalent to one of the varieties in the following list.

\begin{enumerate}
\item[($B_{1}$)] Homogeneous spaces of connected linear algebraic groups with connected stabilizers.
\item[($B_{2}$)] Homogeneous spaces of semi-simple simply connected algebraic groups with abelian stabilizers.
\begin{itemize} 
\item[-]In these first two cases, the base satisfies the hypothesis 3.2 by the work of Borovoi \cite{Borovoi96}.
\end{itemize}

\item[($B_{3}$)] Ch\^atelet surfaces.
\begin{itemize}
\item[-] The base satisfies the hypothesis 3.2 by the work of Colliot-Th\'el\`ene\textendash~Sansuc\textendash~Swinnerton-Dyer \cite[Thm. 8.11]{chateletsurfaces}.
\end{itemize}

\item[($B_{4}$)] Smooth intersections of two quadrics in $\P^{n}_{k(B)}(n\geq5)$ containing a conic defined over $k(B)$. 
\begin{itemize}
\item[-] We refer to the work of Harari \cite[Prop. 5.8]{Harari}, one may find other examples in the same paper.
\end{itemize}
\end{enumerate}

In Theorem \ref{mainThm-geo-int}, fibrations in the following list will satisfy the hypothesis 2. The hypothesis 1 is not easy to check, nevertheless for trivial fibrations as in Corollary \ref{trivialcor} it is automatically satisfied.
\begin{enumerate}
\item[($F_{1}$)] The fibration $X\to B$ is a Severi\textendash Brauer scheme over $B$.
\begin{itemize}
\item[-] All fibres are geometrically integral. The hypothesis 2 is satisfied by  class field theory.
\end{itemize}

\item[($F_{2}$)] The generic fibre of the fibration $X\to B$ is a Ch\^atelet $p$-fold which has an irreducible defining polynomial, with the hypothesis 1 assumed.
\begin{itemize}
\item[-]  The generic fibre is defined by the equation $N_{L|k(B)}(\textup{x})=P(t)$, where $L$ is a cyclic extension of $k(B)$ of prime degree $p$ and $P(t)\in k(B)[t]$ is an irreducible separable polynomial of degree $2p$. For closed points $\theta$ contained in a suitable generalized Hilbertian subset of $B$ corresponding to the finite extensions $L$ and $k(B)[t]/(P(t))$ of $k(B)$, the fibre $X_{\theta}$ is a Ch\^atelet $p$-fold with irreducible defining polynomial. By \cite[Thm. 3.2]{VAV-chatletp-fold}, the Brauer group $\Br X_{\theta}$ is trivial modulo constant. The Brauer\textendash Manin obstruction is the only obstruction to weak approximation for 0-cycles of degree $1$ on $X_{\theta}$, \cite[Thm. 4.1]{CT-Sk-SD}. Hence the hypothesis 2 is satisfied.
\end{itemize}
\end{enumerate}

In Theorem \ref{mainThm-section}, $X\to B$ can be a fibration whose generic fibre $X_{\eta}$ is a $k(B)$-variety birationally equivalent to one of the varieties in the following list. 
\begin{enumerate}
\item[($F_{3}$)] Connected linear algebraic groups over $k(B)$
\begin{itemize}
\item[-]The hypothesis 1.2 is automatically satisfied, and the hypothesis 2 is ensured by the work of Sansuc \cite{Sansuc}. 
\end{itemize}

\item[($F_{4}$)]  $k(B)$-varieties defined by the equation $N_{L|k(B)}(\textup{x})=P(t)$, where $P(t)\in k(B)[t]$ is a separable polynomial and $L$ is a finite cyclic extension of $k(B)$ of degree prime to  $\deg(P(t))$.
\begin{itemize}
\item[-] The assumption on the degree implies that the hypothesis 1.2 is satisfied. The hypothesis 2 is also satisfied by one of the main results in \cite[Thm. 4.1]{CT-Sk-SD}.
\end{itemize}

\item[($F_{5}$)] $k(B)$-varieties defined by the equation $N_{L|k(B)}(\textup{x})=P(t_{1},\ldots,t_{m})$, where $L$ is a finite extension of $k(B)$ of prime degree $p$ and $P(t_{1},\ldots,t_{m})\in k(B)[t_{1},\ldots,t_{m}]$ is a polynomial whose degree with respect to one of the $t_{i}$'s is prime to $p$.
\begin{itemize}
\item[-] The assumption on the degree implies that the hypothesis 1.2 is satisfied. The hypothesis 2 is also satisfied by \cite[Thm. 4.3]{Wei} and \cite[Ex. 3.1]{Liang5}.
\end{itemize}
\end{enumerate}

We have to remark that for the case  $(F_{3})$, one may also apply \cite[Thm. 3]{HarariSMF} combined with the main result of \cite{Liang4} to obtain the exactness of $(\E)$ for $X$. For the case $(F_{1})$, the result can also be deduced by \cite{Liang4}. For $(F_{2})$, $(F_{4})$ and $(F_{5})$, firstly one lacks  unconditional results on the arithmetic for rational points on closed fibres, secondly there may not be rational points on the generic fibre even if it possesses a 0-cycle of degree $1$. Theorems \ref{mainThm-geo-int} and \ref{mainThm-section} are hence crucial to obtain the exactness of $(\E)$. Very recently, Harpaz and Wittenberg prove the exactness of $(\E)$ for the total space of a fibration $X\to B$ without any restriction on the degenerate fibres, \cite{HarWit}. In their paper, the base $B$ can be a smooth projective curve $C$ with finite Tate\textendash Shafarevich group $\sha(\textup{Jac}(C))$, or $\P^{n}$, or $C\times\P^{n}$. It is not clear if the examples above can be covered by their result.

\section{Proofs of the theorems}\label{proofsection}

\subsection{Preliminaries for the proofs}

For the convenience of the reader, we list some well-known statements in this subsection, they will be  used in the proofs of the main results.

\begin{lem}[Moving lemma for 0-cycles]\label{movinglemma0}
Let $X$ be a smooth connected variety defined over an infinite perfect field $k$. Let $X_{0}$ be a nonempty open subvariety of $X$. Then every 0-cycle $z$ of $X$ is rationally equivalent to a 0-cycle $z'$ supported in $X_{0}$.
\end{lem}

\begin{proof}
This is well-known, for an elementary proof see for example \cite[Compl\'ement \S 3]{CT05}.
\end{proof}

\begin{defn}
A 0-cycle $z$ written as $z=\sum n_{P}P$ with distinct closed points $P$ is said to be \emph{separable} if all nonzero multiplicities $n_{P}$ are either $1$ or $-1$.
\end{defn}

Effective separable 0-cycles on $\mathbb{A}^{1}_{k}$ correspond bijectively to separable polynomials over $k$ (up to scalar). Let $f:X\to Y$ be a proper morphism between algebraic varieties and let $z$ be an effective 0-cycle on $X$, then $f_{*}(z)$ is a separable 0-cycle on $Y$ implies that $z$ is separable.

\begin{defn}
Let $d$ be a positive integer and $X$ be a quasi-projective  variety defined over a topological field $k$. Let $\textup{Sym}^{d}X$ be the $d$-th symmetric product of $X$. Effective 0-cycles $z$ and $z'$ of degree $d$ on $X$ correspond to rational points $[z],[z']\in\textup{Sym}^{d}(X)(k)$. One says that $z$ is \emph{sufficiently close} to $z'$ if $[z]$ is sufficiently close to $[z']$ with respect to the topology on $\textup{Sym}^{d}(X)(k)$ induced by $k$.
\end{defn}

For a fixed integer $n>0$, if $k$ is $\mathbb{R},$ $\mathbb{C}$ or any finite extension of $\mathbb{Q}_p$,  the map $\textup{Sym}^{d}(X)(k)\to \CH_{0}(X)/n$ associating an effective  0-cycle to its class is locally constant, \emph{cf.} \cite[Lem. 1.8]{Wittenberg}. In other words $z$ and $z'$ have the same image in $\CH_{0}(X)/n$ if $z$ is sufficiently close to $z'$.

When $X=\P^{1}$ and $k$ is $\mathbb{R}$, $\mathbb{C}$ or a finite extension of $\mathbb{Q}_{p}$, the situation can be written explicitly as follows. Let $z'=\sum P'$  (with $P'$ distinct closed points) be an effective separable 0-cycle and let $z$ be an effective 0-cycle. They are contained in a certain open subset isomorphic to $\mathbb{A}^{1}$ and they are defined respectively by  monic polynomials $f'$ and $f\in k[T]$. We fix an algebraic closure $\bar{k}$ and an embedding  $k(P')\hookrightarrow\bar{k}$ for each $P'$, we can view $P'$ as a $k(P')$-rational point of $\P^{1}$. The 0-cycle $z$ is sufficiently close to $z'$ if and only if  $f$ is sufficiently close to $f'$ with respect to the product topology on $k[T]$ induced by $k$. If it is the case, the 0-cycle $z$ is also separable, therefore $z$ can be written as a sum $\sum P$ of distinct closed points of $\mathbb{A}^{1}$. It is clear if $k$ is archimedean, and it follows from Krasner's lemma if $k$ is non-archimedean, that each $P$ appearing in the support of $z$ corresponds to a unique $P'$ appearing in the support of $z'$ such that they have the same residual field $k(P)=k(P')$ and such that $P$ is sufficiently close to $P'$ as a $k(P')$-rational point. This will allow us to apply the implicit function theorem to fibrations over $\P^{1}$. This will also allow us to apply the continuity of the evaluation of a certain fixed Brauer element at rational points to obtain equalities on its evaluations at sufficiently close effective separable 0-cycles.

\begin{lem}[Relative moving lemma]\label{moving lemma}
Let $\pi:X\to\mathbb{P}^1$ be a dominant morphism between algebraic varieties defined over
$\mathbb{R},$ $\mathbb{C}$ or any finite extension of $\mathbb{Q}_p.$
Suppose that $X$ is smooth.

Then for any effective 0-cycle $z'$ on $X,$ there exists an effective 0-cycle $z$ on $X$
such that $\pi_*(z)$ is separable and
such that $z$ is sufficiently close to $z'.$
\end{lem}

\begin{proof}
Essentially, the statement follows from the implicit function theorem. We find detailed arguments
in \cite[p.19]{CT-Sk-SD} and \cite[p.89]{CT-SD}.
\end{proof}

\begin{lem}[Hilbert's irreducibility theorem for 0-cycles]\label{Hilbert_irred}
Let $S$ be a nonempty finite subset of places of a number field $k$.
Let $\Hil\subset \mathbb{A}^{1}=\P^{1}\setminus\infty$ be a generalized Hilbertian subset of $\P^{1}_{k}$.
For each $v\in S$, let $z_{v}$ be a separable effective 0-cycle of degree $d>0$ with support contained in $\mathbb{A}^{1}$.

Then there exists a closed point $\lambda$ of $\P^{1}$ such that
\begin{itemize}
\item[-] $\lambda\in\Hil$;
\item[-] as a 0-cycle $\lambda$ is sufficiently close to $z_{v}$ for any $v\in S$.
\end{itemize}
\end{lem}

\begin{proof}
It is a particular case of \cite[Lem. 3.4]{Liang1}.
\end{proof}

\begin{lem}\label{shortexactseq}
Let $X$ be a variety defined over a field $k$ of characteristic $0$. Denote by $\overline{X}$ the base change $X\times_{\Spec(k)}\Spec(\bar{k})$.
Suppose that $\H^{0}(\overline{X},\mathbb{G}_{m})=\bar{k}^{*}$ and that $\H^{3}(k,\mathbb{G}_{m})\to\H^{3}(X,\G_{m})$ is injective. Then there exists an exact sequence
$$0\to\H^{1}(k,\Pic\overline{X})\to\Br X/\Br k\to(\Br\overline{X})^{\Gal_{\bar{k}|k}}\to\H^{2}(k,\Pic\overline{X})$$
\end{lem}

\begin{proof}
This is well-known, see for example \cite[Prop. 1.3]{CT-Sk-Crelle2013} and \cite[Prop. 2.2.1]{Harari2}. For the convenience of the reader, we give a proof with details as follows. We refer to  \cite[Appendix B]{MilneEC} for the notation.

Consider the Leray spectral sequence
$$\E_{2}^{p,q}=\H^{p}(k,\H^{q}(\overline{X},\mathbb{G}_{m}))\Rightarrow E^{p+q}=\H^{p+q}(X,\mathbb{G}_{m}).$$
We obtain the exact sequence
$$\Br k\to\ker[\Br X\to\Br\overline{X}]\to\H^{1}(k,\Pic\overline{X})\to\H^{3}(k,\G_{m}),$$
which extends to a \emph{complex}
$\cdots\to\H^{1}(k,\Pic\overline{X})\to\H^{3}(k,\G_{m})\to\H^{3}(X,\G_{m}).$
Even though it is not exact at $\H^{3}(k,\G_{m})$, the injectivity assumption implies the exactness of
$$\Br k\to\ker[\Br X\to\Br\overline{X}]\to\H^{1}(k,\Pic\overline{X})\to0,$$ whence the exactness of
$$0\to\H^{1}(k,\Pic\overline{X})\to\Br X/\Br k\to\Br\overline{X}.$$
Consider the filtration of $E^{2}$, $$E^{2}=\Br X\twoheadrightarrow E_{\infty}^{0,2}=E_{4}^{0,2}\subset E_{3}^{0,2}\subset E_{2}^{0,2}=(\Br\overline{X})^{\Gal_{\bar{k}|k}},$$ where by definition $E_{3}^{0,2}=\ker\left[d_{2}^{0,2}:E_{2}^{0,2}=(\Br\overline{X})^{\Gal_{\bar{k}|k}}\to E_{2}^{2,1}=\H^{2}(k,\Pic\overline{X})\right]$, it remains to show that $E_{4}^{0,2}$, the image of $\Br X$ in $(\Br\overline{X})^{\Gal_{\bar{k}|k}}$, equals to $E_{3}^{0,2}$. This is equivalent to the fact that $d_{3}^{0,2}:E_{3}^{0,2}\to E_{3}^{3,0}$ is the zero map, or even  $E_{3}^{3,0}\twoheadrightarrow E_{4}^{3,0}$ is an isomorphism. Consider the flitration of $E^{3}$, the composition
$$E_{2}^{3,0}=\H^{3}(k,\G_{m})\twoheadrightarrow E_{3}^{3,0}\twoheadrightarrow E_{4}^{3,0}=E_{\infty}^{3,0}\hookrightarrow E^{3}=\H^{3}(X,\G_{m})$$
is injective by assumption, which implies the desired equality $E_{4}^{0,2}=E_{3}^{0,2}$.
\end{proof}

\begin{lem}[Harari {\cite[Lem. 1]{HarariSMF}}]\label{Hilbert-for-rational-points}
Let $B$ be a smooth geometrically integral variety defined over a number field $k$ such that $B(k_{v})\neq\emptyset$ for all places $v\in\Omega_{k}$. Suppose that for $k$-rational points $B$ verifies weak approximation outside a certain finite set $T$ of places. Let $\mathsf{H}$ be a (classical) Hilbertian subset of  $B(k)$.
Let $S\subset\Omega_{k}$ be a finite subset of places and $\{P_{v}\}_{v\in S}$ be a family of local points of $B$ such that: for every finite subset $S'\subset\Omega_{k}\setminus(S\cup T)$ and arbitrary family $\{P'_{v}\}_{v\in S'}$ of local rational points, the union $\{Q_{v}\}_{v\in S\sqcup S'}$ of $\{P_{v}\}_{v\in S}$ and $\{P'_{v}\}_{v\in S'}$ is in the closure of the diagonal image of $B(k)$.

Then there exists a rational point of $B$ contained in $\mathsf{H}$  arbitrarily close to $P_{v}$ for all $v\in S$.
\end{lem}

Suppose that the Brauer\textendash Manin obstruction is the only obstruction to weak approximation for rational points on $B$ and that $\Br B/\Br k$ is finite. Take a finite subset $T\subset\Omega_{k}$ large enough such that there exists a certain set of representatives of $\Br B/\Br k$ whose elements have good reduction outside $T$. If $\{P_{v}\}_{v\in\Omega_{k}}$ is orthogonal to $\Br B$, then for arbitrary finite subset $S\subset\Omega_{k}$ the family $\{P_{v}\}_{v\in S}$ satisfies automatically the assumption of the lemma.

\begin{lem}[Harari's formal lemma {\cite[Cor. 2.6.1]{Harari}} {\cite[Lem. 4.5]{CT-Sk-SD}}]\label{formal-lemma}
Let $X$ be a smooth proper geometrically integral variety defined over a number filed $k$. Let $X_{0}$ be a nonempty open subvariety of $X$ and $\Lambda\subset\Br X_{0}\subset \Br k(X)$ be a finite set of elements of the Brauer group. We note by $B$ the intersection in $\Br k(X)$ of $\Br X$ and the subgroup generated by $\Lambda$.

Let $\delta$ be an integer.  Suppose that for every $v\in\Omega_{k}$, there exists a 0-cycle $z_{v}$ on $X_{k_v}$ of degree $\delta$ supported in  $X_{0v}$ such that the family $\{z_{v}\}_{v\in\Omega_{k}}$ is orthogonal to $B$.

Then for every finite subset $S\subset\Omega_{k}$, there exist a finite subset $S'$ of places of $k$ containing $S$ and for each $v\in S'\setminus S$ a 0-cycle $z'_{v}$ on $X$ of degree $\delta$ and supported in $X_{0v}$ such that
$$\sum_{v\in S}\inv_{v}(\langle z_{v},b\rangle_{v})+\sum_{v\in S'\setminus S}\inv_{v}(\langle z'_{v},b\rangle_{v})=0$$
for all $b\in\Lambda$.
\end{lem}

\subsection{Proofs of the main theorems}

In order to prove Theorem \ref{mainThm-section} we need to establish some lemmas on generalised Hilbertian subsets to compare Brauer groups.
If $k$ is a number field, one can deduce from Hilbert's irreducibility theorem that generalized Hilbertian subsets of geometrically integral varieties are always nonempty. Let $\Hil_{1}$ and $\Hil_{2}$ be generalized Hilbertian subsets of an integral variety, then there exists  a generalized Hilbertian subset $\Hil$ contained in $\Hil_{1}\cap\Hil_{2}$.

\begin{lem}\label{lemHil}
Let $B$ be a geometrically integral variety defined over $k$. Let $\Hil_{B}$ be a generalized Hilbertian subset of $B$. Consider the projections $pr_{B}:B\times \P^{1}\to B$ and $pr_{\P^{1}}:B\times \P^{1}\to \P^{1}$. Then there exists a generalized Hilbertian subset $\Hil_{\P^{1}}$ of $\P^{1}$ satisfying the following property.

\begin{itemize}
\item For each $\lambda\in\Hil_{\P^{1}}$, the $k(\lambda)$-variety $pr_{\P^{1}}^{-1}(\lambda)=B\times\lambda\simeq B_{k(\lambda)}$ contains a generalized Hilbertian subset $\Hil_{\lambda}$ such that the image $pr_{B}(\Hil_{\lambda})\subset B$ is contained in $\Hil_{B}.$
\end{itemize}
\end{lem}

\begin{proof}
The generalized Hilbertian subset $\Hil_{B}$ is defined by a finite \'etale morphism $\rho:Z\to U\subset B$, where $U$ is a non-empty open subset of $B$ and $Z$ is an integral $k$-variety. The function field $k(Z)$ is a finitely generated field extension of $k$, let $k'$ be the algebraic closure of $k$ in $k(Z)$.
The finite \'etale morphism $\P^{1}\times_{\Spec(k)} \Spec(k')\to\P^{1}$ defines a generalized Hilbertian subset $\Hil_{\P^{1}}$ of $\P^{1}.$

For each $\lambda\in\Hil_{\P^{1}}$, the extension $k(\lambda)/k$ is linearly disjoint from $k'/k$, hence the base change $Z_{k(\lambda)}$ is still an integral variety over $k(\lambda)$. Let $\Hil_{\lambda}$ be the generalized Hilbertian subset defined by $\rho_{k(\lambda)}:Z_{k(\lambda)}\to U_{k(\lambda)}\subset B_{k(\lambda)}$. For each $\theta\in\Hil_{\lambda}$, the (connected) fibre $\rho_{k(\lambda)}^{-1}(\theta)$ is a closed point of $Z_{k(\lambda)}$. We can show by contradiction that the fibre of $\rho$ at the closed point $pr_{B}(\theta)$ of $B$ has also to be connected, which signifies that $pr_{B}(\theta)\in\Hil_{B}$ and we have $pr_{B}(\Hil_{\lambda})\subset\Hil_{B}.$
\end{proof}

\begin{lem}[Comparison of Brauer groups]\label{lemBr}
Let $f:X\to B$ be a dominant morphism between proper smooth geometrically integral varieties defined over   $k$.
Assume that the generic fibre $X_{\eta}$ possesses a 0-cycle of degree $1$. Suppose for the geometric generic fibre $X_{\bar{\eta}}$, that the N\'eron\textendash Severi group $\textup{NS}X_{\bar{\eta}}$ is torsion-free, and $\H^{1}(X_{\bar{\eta}},\mathcal{O}_{X_{\bar{\eta}}})=0$, $\H^{2}(X_{\bar{\eta}},\mathcal{O}_{X_{\bar{\eta}}})=0$.

Then there exists a generalized Hilbertian subset $\Hil_{\P^{1}}$ of $\P^{1}$  satisfying the following property.

\begin{itemize}
\item For each $\lambda\in\Hil_{\P^{1}}$, the $k(\lambda)$-variety $pr_{\P^{1}}^{-1}(\lambda)=B\times\lambda\simeq B_{k(\lambda)}$ contains a generalized Hilbertian subset $\Hil_{\lambda}$ such that for each $k(\lambda)$-rational point $\theta$ contained in $\Hil_{\lambda}$ all homomorphisms in the following commutative diagram are isomorphisms
\LARGE
\SelectTips{eu}{12}$$\xymatrix@C=65pt @R=20pt{
\frac{\Br X_{\eta}}{\Br k(B)}\ar[r]^-{Sp_{X\to B,pr_{B}(\theta)}}\ar[d]&\frac{\Br X_{pr_{B}(\theta)}}{\Br k(pr_{B}(\theta))}\ar[d]
\\ \frac{\Br (X\times\lambda)_{\eta_{\lambda}}}{\Br k(\lambda)(B\times\lambda)}\simeq\frac{\Br (X_{k(\lambda)})_{\eta_{\lambda}}}{\Br k(\lambda)(B_{k(\lambda)})}\ar[r]^-{Sp_{X\times\lambda\to B\times\lambda,\theta}}&\frac{\Br(X\times\lambda)_{\theta}}{\Br k(\theta)}\simeq\frac{\Br X_{pr_{B}(\theta)}\otimes k(\theta)}{\Br k(\theta)}
},$$
\normalsize
where $\eta_{\lambda}$ is the generic point of $B\times\lambda$ and
where the horizontal homomorphisms are specializations and the vertical homomorphisms are base changes.
\end{itemize}
\end{lem}

\begin{proof}
By applying \cite[Prop. 3.1.1]{Liang4}, which holds for any base field of characteristic $0$, to the generic fibre $X_{\eta},$ we can choose a finite extension $K'$ of $k(B)$ such that the left vertical homomorphism is an isomorphism once $k(\lambda)(B)$ is linearly disjoint from $K'$ over $k(B)$.
The field $K'$ is a finitely generated extension of $k$, let $k'$ be the algebraic closure of $k$ in $K'$. Any finite extension $l$ of $k$ linearly disjoint from $k'/k$ is then linearly disjoint from $K'/k$, and therefore $l(B)$ is linearly disjoint from $K'$ over $k(B).$ The finite \'etale morphism $\P^{1}\times_{\Spec(k)} \Spec(k')\to\P^{1}$ defines a generalized Hilbertian subset $\Hil_{1}$ of $\P^{1}$ such that for each $\lambda\in\Hil_{1}$ the left vertical homomorphism is an isomorphism.

It suffices to make sure that the two horizontal specializations are isomorphisms for $\theta$ and $\lambda$ in certain suitable generalized Hilbertian subsets that are going to be chosen. First we show this for $Sp_{X\to B,pr_{B}(\theta)}$, and for the lower horizontal homomorphism it is exactly the same argument applied for $X\times\lambda$ instead of $X$. To simplify notation, we replace $pr_{B}(\theta)\in B$ by $\theta\in B$ in (and only in) the next paragraph.

By the restriction-corestriction argument, the existence of a 0-cycle of degree $1$ on the generic fibre $X_{\eta}$ ensures that $\H^{3}(k(B),\mathbb{G}_{m})\to\H^{3}(X_{\eta},\mathbb{G}_{m})$ is injective. Then from the Leray spectral sequence
$$\E_{2}^{p,q}=\H^{p}(k(B),\H^{q}(X_{\bar{\eta}},\mathbb{G}_{m}))\Rightarrow\H^{p+q}(X_{\eta},\mathbb{G}_{m})$$ one deduces an exact sequence by Lemma \ref{shortexactseq},
$$0\to\H^{1}(k(B),\Pic X_{\bar{\eta}})\to\Br X_{\eta}/\Br k(B)\to (\Br X_{\bar{\eta}})^{\Gal_{\overline{k(B)}|k(B)}}\to\H^{2}(k(B),\Pic X_{\bar{\eta}}).$$ 
Take a finite Galois extension $L$ of $k(B)$ such that the generic fibre $X_{\eta}$ possesses a $L$-rational point, such that the absolute Galois group $\Gal_{\overline{L}|L}$  acts trivially on $\Pic X_{\bar{\eta}}$ and $\Br X_{\bar{\eta}}$ (both are finitely generated under the hypotheses on $X_{\bar{\eta}}$), and such that the image of $(\Br X_{\bar{\eta}})^{\Gal_{\overline{k(B)}|k(B)}}$ in $\H^{2}(k(B),\Pic X_{\bar{\eta}})$ lies inside the subgroup $\H^{2}(L|k(B),\Pic X_{\bar{\eta}})$. Then the exact sequence becomes an exact sequence of finite abelian groups
$$0\to\H^{1}(L|k(B),\Pic X_{\bar{\eta}})\to\Br X_{\eta}/\Br k(B)\to (\Br X_{\bar{\eta}})^{\Gal_{L|k(B)}}\to\H^{2}(L|k(B),\Pic X_{\bar{\eta}}).$$
We denote  $X_{\theta}\times_{\Spec(k(\theta))}\Spec(\overline{k(\theta)})$ by $X_{\bar{\theta}}$.
By taking an integral model over a certain sufficiently small non-empty open subset $U$ of $B$, one can specialise each term of the above sequence at each point $\theta\in U$ and one obtains a commutative diagram
\SelectTips{eu}{12}$$\xymatrix@C=10pt @R=20pt{
0\ar[r]&\H^{1}(L|k(B),\Pic X_{\bar{\eta}})\ar[r]\ar[d]&\Br X_{\eta}/\Br k(B)\ar[r]\ar[d]^{Sp_{X\to B,\theta}}&(\Br X_{\bar{\eta}})^{\Gal_{L|k(B)}}\ar[r]\ar[d]&\H^{2}(L|k(B),\Pic X_{\bar{\eta}})\ar[d]
\\0\ar[r]&\H^{1}(l|k(\theta),\Pic X_{\bar{\theta}})\ar[r]&\Br X_{\theta}/\Br k(\theta)\ar[r]&(\Br X_{\bar{\theta}})^{\Gal_{l|k(\theta)}}\ar[r]&\H^{2}(l|k(\theta),\Pic X_{\bar{\theta}})
},$$
where all the vertical maps are specializations at $\theta$.
Moreover, if $\theta$ belongs to the generalized Hilbertian subset $\Hil_{B}$ defined by the finite morphism whose generic fibre corresponds to the finite extension $L/k(B)$, then the specialization $\Spec(l)=\Spec(L)\times_{U}\theta$ of $\Spec(L)$ at $\theta$ is connected. This signifies that $l$ is a field and $\Gal_{L|k(B)}\simeq\Gal_{l|k(\theta)}$.  Replacing by smaller $U$ if necessary,  the specialization maps $\Pic X_{\bar{\eta}}\to \Pic X_{\bar{\theta}}$ and  $\Br X_{\bar{\eta}}\to \Br X_{\bar{\theta}}$ are both isomorphisms, \textit{cf.} \cite[Prop. 3.4.2]{Harari} and \cite[Prop. 2.1.1]{Harari2}. The same argument as for the generic fibre shows that the second line of the diagram is also exact. The isomorphism between Galois groups implies that all the three unnamed vertical maps in the above diagram are isomorphisms. By diagram chasing we prove that $Sp_{X\to B,\theta}$ is also an isomorphism.

In the previous paragraph we have proved that the specialization $Sp_{X\to B,pr_{B(\theta)}}$ is an isomorphism as long as  $pr_{B}(\theta)$ is contained in $\Hil_{B}$. By Lemma \ref{lemHil}, there exist a generalized Hilbertian subset $\Hil_{2}$ of $\P^{1}$ and $\Hil_{\lambda}\subset B\times\lambda$ such that for all $\lambda\in\Hil_{2}$ we have $pr_{B}(\Hil_{\lambda})\subset \Hil_{B}$, hence for such a $\lambda$ and $\theta\in\Hil_{\lambda}$ the specialization $Sp_{X\to B,pr_{B}(\theta)}$ is an isomorphism.
Moreover, by the same argument as in the previous paragraph applied to the $k(\lambda)$-variety $B\times\lambda$, by shrinking $\Hil_{\lambda}$ if necessary, we may also ensure that $Sp_{X\times\lambda\to B\times\lambda,\theta}$ is  an isomorphism when $\theta\in\Hil_{\lambda}$.

Now we take a  generalized Hilbertian subset $\Hil_{\P^{1}}\subset \Hil_{1}\cap\Hil_{2}$ of $\P^{1}$, it verifies the desired property.
\end{proof}

We are now ready to give the proof of  Theorem \ref{mainThm-section}. In \cite{Liang4}, to obtain results on $X$ the author passed from $X$ to $X\times\P^{1}$, here we extend this idea to a morphism $X\to B$.

\begin{proof}[Proof of Theorem \ref{mainThm-section}]
We will prove the exactness of $(\E)$, the assertion in the theorem for weak approximation is simultaneously proved. In fact, under the hypothesis 4, the exactness will be reduced to the forthcoming statement $(P'_{S})$ for $X'=X\times\P^{1}$. As $\CH_{0}(X'_{k_{v}})\to\CH_{0}(X'_{k_{v}})/n$ factors through $\CH_{0}(X'_{k_{v}})/2n$, the statement $(P'_{S})$ implies  the assertion  on weak approximation for 0-cycles on $X'$, and hence for $X$. For the assertion on the Hasse principle, one needs only take care of the existence on 0-cycles. 
In the forthcoming arguments proving $(P'_{S})$, if one ignores the part concerning images of 0-cycles in $\CH_{0}(X'_{k_{v}})/2n$, then one will get a proof for the assertion on the Hasse principle. For this purpose, one needs the hypothesis 2 on the Hasse principle (instead of on weak approximation) for 0-cycles on fibres.

Since $\Br X\to\Br X\times\P^1$ is an isomorphism, it suffices to prove the exactness of $(\E)$ for $X'=X\times\P^{1}$.
As $\deg:\CH_{0}(X\otimes\overline{k(X)})\to\Z$ is injective, there exists a non-zero integer $N$ such that for all extensions $F$ of $k$ the kernel $\ker\left[\deg:\CH_{0}(X_{F})\to\Z\right]$ and the cokernel  $\coker\left[\deg:\CH_{0}(X_{F})\to\Z\right]$ are annihilated by $N$. In fact, for the cokernel, the argument is standard using the projection formula; and for the kernel see the proof of \cite[Prop. 11]{CT05}. Note that $\Pic X_{\bar{k}}$ is torsion-free, we obtain $\H^{1}(X_{\bar{k}},\Q/\Z)=0$, hence $\H^{1}(X'_{\bar{\eta}}\otimes{\overline{k(X')}},\Q/\Z)=\H^{1}(X_{\overline{k(X')}},\Q/\Z)=0$ by the proper base theorem \cite[Cor. VI.2.6]{MilneEC}. According to \cite[Thm. 2.1]{Wittenberg}, we find that $\deg:\CH_{0}(X_{k_v})\to\Z$, which is identified with $\CH_{0}(X'_{k_v})\to\CH_{0}(\P^{1}_{k_{v}})$, is injective for almost all places $v$ of $k$. 

Consider the composition $pr_{\P^{1}}\circ F$ of
$F=f\times id: X'=X\times \P^{1}\to B'=B\times\P^{1}$ with $pr_{\P^{1}}: B'=B\times\P^{1}\to\P^{1}$. In order to prove the exactness of $(\E)$ for $X'$ we apply \cite[Prop. 3.1]{Wittenberg} (see also \cite[\S 3.5]{Liang4} for a simplified argument for this special case: a fibration over $\P^1$), it suffices to verify the following property for all finite subset $S\subset \Omega_{k}$ of places of $k$.
\begin{enumerate}
\item[$(P_S)$]
Let $\{z_v\}_{v\in\Omega_k}$ be a family of 0-cycles of degree $\delta$ on $X'$. If it is orthogonal to $\Br X'$, then for all integer $n>0$, there exists a 0-cycle $z$ of $X'$ of degree $\delta$, such that for all $v\in S$ we have $z=z_v$ in $\CH_0(X'_{k_v})/n$ if $v$ is non-archimedean and $z=z_v+N_{\bar{k}_v/k_v}(u_v)$ in $\CH_0(X'_{k_v})$ for a certain $u_v\in \CH_0({X'_{\bar{k}_v}})$ if $v$ is real.
\end{enumerate}
The property $(P_S)$ is implied by the following property, since for a real place  $\CH_{0}(X'_{k_{v}})\to\CH_{0}(X'_{k_{v}})/N_{\bar{k}_{v}|k_{v}}\CH_{0}(X'_{\bar{k}_{v}})$ factors through $\CH_{0}(X'_{k_{v}})/2n$.
\begin{enumerate}
\item[$(P'_S)$]
Let $\{z_v\}_{v\in\Omega_k}$ be a family of 0-cycles of degree $\delta$ on $X'$. If it is orthogonal to $\Br X'$, then for all integer $n>0$, there exists a 0-cycle $z$ of $X'$ of degree $\delta$, such that for all $v\in S$ we have $z=z_v$ in $\CH_0(X'_{k_v})/2n.$
\end{enumerate}

We are going to prove $(P'_S)$ for arbitrary degree $\delta$.

Note that $B$ is geometrically integral, by enlarging $S$ if necessary, according to Lang\textendash Weil estimation and Hensel's lemma, we may assume that $B$ has $k_{v}$-rational points for all $v\in \Omega_{k}\setminus S$.
Under the geometric hypothesis on $B$ the group $\Br B_{\bar{k}}$ is finite and $\Pic B_{\bar{k}}$ is torsion-free, we deduce the finiteness of $\Br B/\Br k$ via the Leray spectral sequence. We fix a finite set of representatives $\Gamma_B\subset \Br B\simeq\Br B'$.
By the argument of good reduction, we may also assume that the finite set $S$ is large enough such that local evaluations of elements $b_{B}\in\Gamma_B$ are all $0$ for $v\notin S$.

Similarly, the quotient $\Br X_{\eta}/\Br k(B)$ is finite, we choose a finite set $\Gamma_{X}\subset\Br X_{\eta} \subset\Br k(X)$ of representatives. By hypothesis, let $z_{0}$ be a 0-cycle of degree $1$ on $X_{\eta}$. We write $z_{0}=\sum_{j}n_{j}R_{j}(\eta)$  where $R_{j}(\eta)$ is a closed point of $X_{\eta}$ of residual field $K_{j}$. Denote by $d_{j}$ the degree $[K_{j}:k(B)]$. There exist a nonempty open subset $U$ of $B$ and an \'etale morphism $Z_{j}\to U$ of degree $d_{j}$ for each $j$, where $Z_{j}$ is an integral closed subvariety  of $X\times_{B}U$ of function field $K_{j}$. We pose $b_{X}^{0}=\sum_{j}n_{j}\cores_{K_{j}|k(B)}(b_{X}(R_{j}(\eta)))\in\Br k(B)$. Replacing $b_{X}$ by $b_{X}-b^{0}_{X}$ if necessary, we may assume that
$$\sum_{j}n_{j}\cores_{K_{j}|k(B)}(b_{X}(R_{j}(\eta)))=0\in\Br k(B).\leqno(\star)$$
 Let $X_{0}$ be a nonempty open subset of $X$ such that $\Gamma_{X}\subset\Br X_{0}\subset\Br k(X)$. By shrinking $U$ and $X_{0}$ if necessary, we may assume that $f(X_{0})=U$.

By augmenting  $L$ if necessary, according to \cite[Prop. 3.1.1]{Liang4} we may assume that for any $K\in\K_{L}$ the base extension $\Br B/\Br k\to\Br B_{K}/\Br K$ is an isomorphism. The finite \'etale morphism $\P^{1}_{L}\to\P^{1}$ defines a generalized Hilbertian subset $\Hil_{0}$ of $\P^{1}$. As long as $\lambda\in\Hil_{0}$, the homomorphism $\Br B/\Br k\to\Br B\times\lambda/\Br k(\lambda)$ is an isomorphism and $\Hil_{k(\lambda)}\subset B_{k(\lambda)}$ is defined as in the hypothesis.

We begin with a family $\{z_v\}_{v\in\Omega_k}$ of 0-cycles of degree $\delta$ orthogonal to $\Br X'$. We fix an (even) integer $n$, and we are looking for a global 0-cycle $z$ of degree $\delta$ having the same image in $\CH_{0}(X'_{k_v})/n$ for $v\in S$.
After Lemma \ref{movinglemma0}, we may suppose that the 0-cycles $z_{v}$ are all supported in $X_{0}\times\P^{1}$.
Firstly, we have for any $b_{B}\in \Gamma_B$
$$\sum_{v\in\Omega_{k}}\inv_{v}(\langle z_{v},F^{*}(b_{B})\rangle_{v})=0,$$ then
$$\sum_{v\in S}\inv_{v}(\langle z_{v},F^{*}(b_{B})\rangle_{v})=0.$$
According to Harari's formal lemma (Lemma \ref{formal-lemma}), by augmenting $S$ (which will never change the above equality because by the choice of $S$ the evaluations at newly added places are always $0$) we have
$$\sum_{v\in S}\inv_{v}(\langle z_{v},b_{X}\rangle_{v})=0$$ for all $b_{X}\in\Gamma_{X}$.

Let $m$ be a positive integer which annihilates all elements in $\Gamma_B$ and $\Gamma_{X}$. Fix a closed point $P$ of $X'$, denote by $\delta_{P}=[k(P):k]$ the degree of $P$.

For each $v\in S$, we write $z_{v}=z_{v}^{+}-z_{v}^{-}$ where $z_{v}^{+}$ and $z_{v}^{-}$ are effective 0-cycles with disjoint supports.
We pose $z_{v}^{1}=z_{v}+mn\delta_{P}z_{v}^{-}=z_{v}^{+}+(mn\delta_{P}-1)z_{v}^{-}$, then $\deg(z_{v}^{1})\equiv\delta\mod{mn\delta_{P}}$  and they are all effective 0-cycles. We add to each $z_{v}^{1}$  a suitable multiple of de 0-cycle $mnP_{v}$ where $P_{v}=P\times_{\Spec(k)}\Spec(k_{v})$ and we obtain $z_{v}^{2}$ of the same degree $\Delta\equiv\delta\mod{mn\delta_{P}}$ for all $v\in S$. According to the choice of $m$, we have $\langle z_{v}^{2}, F^{*}(b)\rangle_{v}=\langle z_{v}, F^{*}(b)\rangle_{v}$ for any $v\in S$ and $b\in\Gamma_B$. We apply Lemma \ref{moving lemma} to $pr_{\P^{1}}\circ F: X'\to\P^{1}$ and obtain an effective 0-cycle $z_{v}^{3}$ sufficiently close to $z_{v}^{2}$ such that $(pr_{\P^{1}}\circ F)_{*}(z_{v}^{3})$ is separable. A fortiori, $F_{*}(z_{v}^{3})$ are also separable. By the continuity of the local evaluation, we have $\langle z_{v}^{3},F^{*}(b)\rangle_{v}=\langle z_{v}, F^{*}(b)\rangle_{v}$. We also check that $z_{v}$, $z_{v}^{1}$, $z_{v}^{2}$ and $z_{v}^{3}$ have the same image in $\CH_{0}(X'_{k_v})/n$.

We choose a rational point $\infty\in\P^{1}(k)$ outside the supports of $(pr_{\P^{1}}\circ F)_{*}(z^{3}_{v})$ for all $v\in S$.
We choose generalized Hilbertian subsets $\Hil_{\P^{1}}\subset\P^{1}$ and $\Hil_{\lambda}\subset B_{k(\lambda)}$ for $\lambda\in\Hil_{\P^{1}}$ such that the property in Lemma \ref{lemBr} is satisfied. Without lost of generality we may assume that $\Hil_{\P^{1}}\subset\Hil_{0}\cap(\P^{1}\setminus\infty)$ and $\Hil_{\lambda}\subset\Hil_{k(\lambda)}\cap U_{k(\lambda)}$ if $\lambda\in\Hil_{\P^{1}}$.

By Hilbert's irreducibility for 0-cycles (Lemma \ref{Hilbert_irred}) applied to $pr_{\P^{1}}\circ F:X'\to\P^{1}$, we find a closed point $\lambda\in\Hil_{\P^{1}}$ such that $\lambda$ is sufficiently close to $(pr_{\P^{1}}\circ F)_{*}(z_{v}^{3})$. More precisely, we write
$\lambda_v=\lambda\times_{\mathbb{P}^1}{\mathbb{P}_{k_{v}}^1}=\bigsqcup_{w\mid v,w\in\Omega_{k(\lambda)}}\Spec(k(\lambda)_w)$
for $v\in\Omega_k,$ the image of $\lambda$ in $Z_0(\mathbb{P}^1_{k_{v}})$ is written as
$\lambda_v=\sum_{w\mid v,w\in\Omega_{k(\lambda)}}P_w$ where $P_w=\Spec(k(\lambda)_w)$ is a closed
point of ${\mathbb{P}^1_{k_{v}}}$ of residual field $k(\lambda)_w.$
For each $v\in S,$ the 0-cycle $\lambda_v$ is sufficiently close to $(pr_{\P^{1}}\circ F)_*(z^3_v),$ where
the effective separable 0-cycle $(pr_{\P^{1}}\circ F)_*(z^3_v)$ is written as $\sum_{w\mid v,w\in\Omega_{k(\lambda)}}Q_w$ with
distinct $Q_w$'s.
Then  $k(\lambda)_w=k_v(P_w)=k_v(Q_w),$ and $P_w$ is sufficiently close to $Q_w\in \mathbb{P}^1_{k_{v}}(k(\lambda)_w).$
And we know that
$z^3_v$ is written as $\sum_{w\mid v,w\in\Omega_{k(\lambda)}}M^0_w$ with $k_v(M^0_w)=k(\lambda)_w$ and
$M^0_w\in X'_{k_v}(k(\lambda)_w)$ is situated
on the fiber of $pr_{\P^{1}}\circ F$ at the closed point $Q_w.$
The implicit function theorem implies that there exists a smooth $k(\lambda)_w$-point $M_w$ on the fiber
$(pr_{\P^{1}}\circ F)^{-1}(\lambda)$ sufficiently close
to $M^0_w$ for every $w\in S\otimes_kk(\lambda).$ Then the closed point $M_w$ and $M_w^0$ have the same
image in $CH_0(X_{k_v})/n$.

As in the proof of Theorem \ref{mainThm-geo-int}, by the continuity of the local evaluation, for any $b_{B}\in \Gamma_{B}$ we obtain on $B'$ the equality
$$\sum_{w\in S\otimes_kk(\lambda)}\inv_w( \langle F(M_w),b_{B} \rangle_{k(\lambda)_w})=0.$$
And similarly, on $X'$,  for $b_{X}\in\Gamma_{X}$
$$\sum_{w\in S\otimes_kk(\lambda)}\inv_w( \langle M_w,b_{X} \rangle_{k(\lambda)_w})=0.$$

On the $k(\lambda)$-variety $pr_{\P^{1}}^{-1}(\lambda)=B\times\lambda\simeq B_{k(\lambda)}$, we fix a $k(\lambda)_{w}$-rational point $N_{w}$ for each $w\in\Omega_{k(\lambda)}\setminus S\otimes_{k}k(\lambda)$, this is possible by the choice of $S$. We set $N_{w}=F(M_{w})$ for $w\in S\otimes_{k}k(\lambda)$. Then we have the equality for all $b_{B}\in\Gamma_{B}$
$$\sum_{w\in \Omega_{k(\lambda)}}\inv_w( \langle N_w,b_{B} \rangle_{k(\lambda)_w})=0.$$

By abuse of notation, here $b_{B}$ denotes also its image under the restriction $\Br B'\to\Br B\times\lambda$. Then by functoriality,  the above equality is viewed as the Brauer\textendash Manin pairing on $B\times \lambda$. Recall that $\lambda\in\Hil_{\P^{1}}\subset\Hil_{0}$, the base extension $\Br B/\Br k\to\Br B\times\lambda/\Br k(\lambda)$ is an isomorphism. Therefore $\{N_{w}\}_{w\in\Omega_{k(\lambda)}}$ is orthogonal to the Brauer group of $B\times\lambda$.

By hypothesis, weak approximation gives us a $k(\lambda)$-rational point $\theta$ on $B\times\lambda$ sufficiently close to $N_{w}=F(M_{w})$ for all $w\in S\otimes_{k}k(\lambda)$,
we can also require that the fibre $F^{-1}(\theta)$ is smooth.  Moreover, by Lemma \ref{Hilbert-for-rational-points} we can choose such a rational point $\theta$ contained in the (classic) Hilbertian subset $\Hil_{\lambda}\cap(B\times\lambda)(k(\lambda)).$
By implicit function theorem $F^{-1}(\theta)$ has $k(\lambda)$-points $M_{w}^{\theta}$ sufficiently close to $M_{w}$ for $w\in S\otimes_{k}k(\lambda)$. By the continuity of the Brauer\textendash Manin pairing, we have
$$\sum_{w\in S\otimes_kk(\theta)}\inv_w( \langle M^{\theta}_w,b_{X} \rangle_{k(\theta)_w})=0$$
for all $b_{X}\in\Gamma_{X}$.
We consider the specialization $R_{j}(\theta)=(Z_{j})_{k(\lambda)}\times_{U_{k(\lambda)}}\theta=Z_{j}\times_{U}\theta$ of $R_{j}(\eta)$  at the $k(\lambda)$-rational point $\theta$ of $B\times\lambda\simeq B_{k(\lambda)}$. Then $\sum_{j}n_{j}R_{j}(\theta)$ is a global 0-cycle of degree $1$ on the fibre $F^{-1}(\theta)=(B\times\lambda)_{\theta}$. The equality $(\star)$ implies that
$$\sum_{j}n_{j}\cores_{k(R_{j}(\theta))|k(\theta)}(b_{X}(R_{j}(\theta)))=0\in\Br k(\theta).$$
For $w\in \Omega_{k(\lambda)}\setminus S\otimes_{k}k(\lambda)$,  we set $M_{w}^{\theta}$ to be the local 0-cycle of degree $1$
$$\sum_{j}n_{j}R_{j}(\theta)\times_{\Spec(k(\theta))}\Spec(k(\theta)_{w})$$
then $\inv_w( \langle M^{\theta}_w,b_{X} \rangle_{k(\theta)_w})=0$, and finally we obtain
$$\sum_{w\in \Omega_{k(\theta)}}\inv_w( \langle M^{\theta}_w,b_{X} \rangle_{k(\theta)_w})=0.$$
By functoriality of the Brauer\textendash Manin pairing, it can be viewed as an equality on the $k(\theta)$-variety $F^{-1}(\theta)$.
Since $\theta\in\Hil_{\lambda}$, by comparison of Brauer groups (Lemma \ref{lemBr}) $\Gamma_{X}$ maps surjectively onto $\Br F^{-1}(\theta)/\Br k(\theta)$, and by hypothesis there exists a global 0-cycle $z'$ of degree $1$ on the $k(\theta)$-variety $F^{-1}(\theta)$ having the same image as $M_{w}^{\theta}$ in $\CH_{0}(F^{-1}(\theta)_{{w}})/n$ for all $w\in S\otimes_{k}k(\theta)$. Viewed as a 0-cycle of $X'$, the 0-cycle $z'$ is of degree $\Delta\equiv\delta\mod mn\delta_{P}$, by subtracting  a suitable multiple of  $P$, we obtain a 0-cycle $z$ of degree $\delta$. We verify that $z$ and $z_{v}$ have the same image in  $\CH_{0}(X'_{k_v})/n$ for all $v\in S$ which terminates the proof.
\end{proof}

\begin{proof}[Sketch of the proof of Theorem \ref{mainThm-geo-int}]
By Remark \ref{rmksplitness}, we can replace the hypothesis 1 by $1'$. The proof is then very similar to the proof of Theorem \ref{mainThm-section}. In Theorem \ref{mainThm-section}, we suppose the hypothesis 1.2: the generic fibre $X_{\eta}$ possesses a 0-cycle of degree $1$, which is used in two purposes. Firstly, as in Remark \ref{rmksplitness}, this condition implies hypothesis $1'$, which ensures the existence of local 0-cycles of degree $1$ for almost all places on  a selected closed fibre. Secondly, this condition also allows us to compare the Brauer group of the selected closed fibre with the Brauer group of the generic fibre. The hypothesis 1.1  also serves for this second purpose. But in  Theorem \ref{mainThm-geo-int}, closed fibres are supposed to satisfy weak approximation, we do not need any more such a comparison of  Brauer groups. Hence the same proof (with such a comparison ignored) fits for  Theorem \ref{mainThm-geo-int}. Moreover, since we do not need the orthogonality of  local 0-cycles to elements of the Brauer group  of the generic fibre, we can even simplify the proof without the application of Harari's formal lemma \ref{formal-lemma}.
\end{proof}

\small


\bibliographystyle{alpha}
\bibliography{mybib1}
\end{document}